\documentclass[11pt]{article}
\usepackage{amsmath,amssymb,amsthm}
\usepackage{booktabs}
\usepackage{xcolor}
\usepackage{tikz}
\usepackage[margin=1in]{geometry}
\allowdisplaybreaks[4]

\usepackage{authblk}
\usepackage[hidelinks]{hyperref}

\numberwithin{equation}{section}

\newtheorem{theorem}{Theorem}[section]
\newtheorem{lemma}[theorem]{Lemma}
\newtheorem{proposition}[theorem]{Proposition}
\newtheorem{corollary}[theorem]{Corollary}
\newtheorem{definition}[theorem]{Definition}
\newtheorem{remark}[theorem]{Remark}

\newcommand{\ser}{\mathbin{\oplus_S}}
\newcommand{\paral}{\mathbin{\oplus_P}}

\begin{document}
\title{The multiplicative Merino--Welsh inequality for chordal graphs}

\author[1,2]{Shaohan Xu\thanks{Corresponding author. Email: xush0928@163.com.}}

\author[3]{Fengming Dong\thanks{Email: fengming.dong@nie.edu.sg and donggraph@163.com.}}

\author[1,2]{Kexiang Xu\thanks{Email: kexxu1221@126.com.}}

\affil[1]{\footnotesize School of Mathematics, Nanjing University of Aeronautics and Astronautics, Nanjing, Jiangsu, 210016, PR China}

\affil[2] {\footnotesize MIIT Key Laboratory of Mathematical Modelling and High Performance Computing of Air  Vehicles, Nanjing, Jiangsu, 210016, PR
China }

\affil[3]{\footnotesize National Institute of Education, Nanyang Technological University, Singapore}

\date{}

\maketitle

\begin{abstract}
For a connected graph $G$, let $\tau(G)$, $\alpha(G)$ and $\alpha^*(G)$ denote its numbers of spanning trees, acyclic orientations and totally cyclic orientations, respectively. We prove that $\alpha(G)\alpha^*(G)\geq\tau(G)^2$ for every connected bridgeless chordal graph.
\end{abstract}

\par
\medskip

\section{Introduction}
All graphs are finite and loopless. Graphs in the main theorem are simple. Auxiliary graphs may have parallel edges. Let $G=(V(G),E(G))$ be a graph with vertex set $V(G)$ and edge set $E(G)$. A \emph{spanning tree} of $G$ is a connected acyclic spanning subgraph of $G$. An \emph{orientation} of $G$ is an assignment of a direction to each edge of $G$. An
orientation is \emph{acyclic} if it contains no directed cycle, and it is \emph{totally cyclic} if every directed edge belongs to a directed cycle.

For a connected graph $G$, let $\tau(G)$, $\alpha(G)$ and $\alpha^*(G)$ denote its numbers of spanning trees, acyclic orientations and totally cyclic orientations, respectively. A \emph{bridge} is an edge whose deletion increases the number of connected components. Merino and Welsh~\cite{MW} conjectured that every connected loopless bridgeless graph satisfies
\begin{equation}\label{eq:original}
\max\{\alpha(G),\alpha^*(G)\}\geq\tau(G).
\end{equation}
Its additive and multiplicative forms, discussed by Conde and Merino~\cite{CM}, are
\begin{equation*}
\alpha(G)+\alpha^*(G)\geq2\tau(G),\qquad \alpha(G)\alpha^*(G)\geq\tau(G)^2.
\end{equation*}
The multiplicative inequality implies the additive inequality by the arithmetic-geometric mean inequality, and the additive inequality implies \eqref{eq:original}.  A bridge lies on no cycle, so a graph with a bridge has $\alpha^*(G)=0$. A graph with a loop has no acyclic orientation. Thus the loopless and bridgeless hypotheses are necessary for the multiplicative inequality when $G$ is connected.

These counting functions are evaluations of the Tutte polynomial. To specify the convention, for a graph $G$ and $A\subseteq E(G)$, let $k(A)$ denote the number of connected components of the spanning subgraph $(V(G),A)$. Then the \emph{Tutte polynomial} of $G$ is defined by
\[
T_G(x,y)=\sum_{A\subseteq E(G)} (x-1)^{k(A)-k(E(G))}(y-1)^{|A|-|V(G)|+k(A)}.
\]
For a connected $G$,
\begin{equation}\label{eq:tutte}
\tau(G)=T_G(1,1),\qquad \alpha(G)=T_G(2,0),\qquad \alpha^*(G)=T_G(0,2).
\end{equation}
See Ellis-Monaghan and Merino~\cite{EM} for background. For disconnected graphs, the orientation counts in \eqref{eq:tutte} remain valid, whereas $T_G(1,1)$ then counts spanning forests with one tree in each connected component.

The Merino--Welsh conjecture has been studied for a number of graph classes. Conde and Merino~\cite{CM} verified the conjecture for threshold graphs and proposed chordal graphs as a class for further study. Thomassen~\cite{Thomassen} obtained inequalities for several sparse and dense families; in particular, every connected simple graph of maximum degree at most three satisfies $\tau(G)\leq\alpha(G)$. Noble and Royle~\cite{NR} proved the multiplicative inequality for loopless bridgeless series-parallel graphs by introducing auxiliary counting parameters for two-terminal graphs. Recently, Chen and Xie~\cite{ChenXie} obtained quantitative Merino--Welsh inequalities for graph joins.

The analogous matroid conjecture asserts that
\[
T_M(2,0)T_M(0,2)\geq T_M(1,1)^2
\]
for every matroid $M$ without loops and coloops. Jackson~\cite{Jackson} proved the weaker inequality
\[
T_M(3,0)T_M(0,3)\geq T_M(1,1)^2,
\]
and further Tutte polynomial inequalities for matroids were obtained by Ch\'avez-Lomel\'i, Merino, Noble and Ram\'irez-Ib\'a\~nez~\cite{CLMNR}. The multiplicative conjecture for matroids  was subsequently proved for lattice path matroids \cite{KMR} and  split matroids \cite{FS}. More recently, Beke, Cs\'aji, Csikv\'ari and Pituk~\cite{BCPfalse} showed that the matroidal version of the Merino--Welsh conjecture is false in general. Their work on the permutation Tutte polynomial~\cite{BCPperm} also reduced the constant $3$ in Jackson's inequality to $2.9243$, and Csikv\'ari~\cite{Csikvari} subsequently reduced it to $2.35$. In contrast, the original conjecture for graphs remains open in general.

Chordal graphs form a natural class for studying the Merino--Welsh conjecture. They admit perfect elimination orderings and clique tree decompositions, yet may contain cliques of arbitrarily large order.  In particular, chordal graphs need not be series-parallel. Thus the result of Noble and Royle \cite{NR} does not cover this class. In this paper, we prove the multiplicative Merino--Welsh inequality for every connected bridgeless chordal graph.

\begin{theorem}\label{thm:main}
Every connected bridgeless chordal graph $G$ satisfies
\[
\alpha(G)\alpha^*(G)\geq\tau(G)^2.
\]
Consequently, it also satisfies the additive and original Merino--Welsh inequalities.
\end{theorem}

Denote by $K_{k}$ the \emph{complete graph} of order $k$. We use the convention $\tau(K_1)=\alpha(K_1)=\alpha^*(K_1)=1$, so the one-vertex case is included in Theorem~\ref{thm:main}. Since all graphs considered in the main theorem are simple, the loopless hypothesis is automatic.

The paper is organized as follows. Section~\ref{sec:prelim} reviews relevant concepts and results in graph theory and introduces the counting parameters. Section~\ref{sec:networks} records a spanning tree formula for unions of two complete graphs and proves a comparison for spanning tree counts under attachments. Section~\ref{sec:interface} proves a quantitative inequality for nontrivial rooted $2$-trees with their root edges deleted and applies it to edge decorations.
Section~\ref{sec:cliques} gives sufficient conditions for complete graph attachments with edge decorations to preserve the multiplicative Merino--Welsh inequality, treating $K_4$ attachments along an edge and along a triangle separately. Section~\ref{sec:mainproof} proves the $2$-connected case by induction on the number of maximal cliques of order at least four and completes the proof of Theorem~\ref{thm:main} by block decomposition.

\section{Preliminaries}\label{sec:prelim}
\subsection{Chordal graphs and clique trees}\label{sec:structure}
For a graph $G$ and a subset $X\subseteq V(G)$, the \emph{induced subgraph} $G[X]$ has vertex set $X$ and contains all edges of $G$ with both ends in $X$. For a vertex $v\in V(G)$, let $G-v$ denote the graph obtained by deleting $v$ and all edges incident with $v$; equivalently, $G-v=G[V(G)\setminus\{v\}]$. For $e\in E(G)$, let $G-e$ denote the graph obtained by deleting $e$ and retaining all vertices.  For an edge $e\notin E(G)$ joining two distinct vertices of $G$, let $G+e$ denote the graph obtained by adding $e$. For a vertex $v$ of a simple graph $G$, its \emph{neighborhood} is denoted by $N_G(v)$ and its degree is $d_G(v)=|N_G(v)|$. For graphs $G_1$ and $G_2$, their \emph{union} $G_1\cup G_2$ is the graph defined by
\[
V(G_1\cup G_2)=V(G_1)\cup V(G_2),\qquad E(G_1\cup G_2)=E(G_1)\cup E(G_2).
\]
Each common edge is retained only once. Distinct parallel edges remain distinct. The graph $G$ is \emph{$2$-connected} if it has at least three vertices and $G-v$ is connected for every $v\in V(G)$.

A \emph{clique} is a vertex set inducing a complete graph. A \emph{maximal clique} is one not properly contained in another clique; it need not have maximum size.  When no ambiguity arises, we identify a clique with the complete subgraph it induces. A graph is \emph{chordal} if it has no induced cycle of length at least four. A vertex is \emph{simplicial} if its neighborhood is a clique. A \emph{perfect elimination ordering} is an ordering $v_1,v_2,\ldots,v_n$ in which $v_i$ is simplicial in $G[\{v_i,v_{i+1},\ldots,v_n\}]$ for each $i$.

Let $\mathcal C(G)$ denote the set of maximal cliques of a connected chordal graph $G$. The maximal cliques in $\mathcal C(G)$ cover all edges of $G$. A \emph{clique tree} of $G$ is a tree $\mathcal T$ with vertex set $\mathcal C(G)$, satisfying the \emph{running-intersection property}: for every vertex $v\in V(G)$, the nodes in $\mathcal T$ corresponding to maximal cliques containing $v$ induce a connected subtree of $\mathcal T$. The edge set $E(\mathcal T)$ is chosen subject to this condition and need not be unique. In particular, two intersecting maximal cliques need not be adjacent in a given clique tree.  We fix a clique tree whenever one is needed.

We use the standard facts that a graph is chordal if and only if it has a perfect elimination ordering, and that every connected chordal graph admits a clique tree. See Blair and Peyton~\cite{BP} for  these structural properties. The theory of chordal graphs goes back to Dirac~\cite{Dirac}.

\begin{lemma}[\cite{BP}]\label{lem:separators}
Let $G$ be $2$-connected and chordal, and let $C$ and $P$ be adjacent nodes of a clique tree of $G$. Then $S=C\cap P$ has at least two vertices and is a proper subset of both $C$ and $P$.
\end{lemma}

\begin{lemma}[\cite{BP}]\label{lem:clique-tree-cut}
Let $\mathcal T$ be a clique tree of a connected chordal graph $G$, and let $C$ and $P$ be adjacent nodes. Let $\mathcal T_1$ and $\mathcal T_2$ be the components of $\mathcal T-CP$ containing $C$ and $P$, respectively. For $i=1,2$, let $V_i$ be the union of the maximal cliques corresponding to the nodes of $\mathcal T_i$. Then $V_1\cap V_2=C\cap P$, and no edge of $G$ joins a vertex of $V_1\setminus(C\cap P)$ to a vertex of $V_2\setminus(C\cap P)$.
\end{lemma}

The chordality assertion in the following lemma follows from \cite[Proposition~5.5.1]{Diestel}, while the $2$-connectedness assertion follows directly from the definition.
\begin{lemma}\label{lem:gluing}
Let $G_1$ and $G_2$ be chordal graphs such that $V(G_1)\cap V(G_2)=S$, where $S$ is a clique in both graphs. Then $G_1\cup G_2$ is chordal. If both $G_1$ and $G_2$ are $2$-connected and $|S|\geq2$, then $G_1\cup G_2$ is $2$-connected.
\end{lemma}

\begin{definition}\label{def:2tree}
A \emph{$2$-tree} is obtained from $K_2$ by repeatedly adding a new vertex adjacent precisely to the two ends of an existing edge. A \emph{rooted $2$-tree} is a $2$-tree with a specified edge $st$, called its \emph{root edge}. It is \emph{nontrivial} if it is not $K_2$.
\end{definition}

Every nontrivial $2$-tree is $2$-connected and chordal, and all of its maximal cliques are triangles. Conversely, every $2$-connected chordal graph whose maximal cliques are all triangles is a $2$-tree. Indeed, root a clique tree at any maximal clique and order its nodes so that each parent precedes its children. By the running-intersection property, the intersection of each new triangle with the preceding union is exactly its intersection with its parent.  By Lemma~\ref{lem:separators}, this intersection is an edge, so each new triangle contributes  one new vertex. Given an edge $st$, choose a triangle containing $st$ as the root and first add its third vertex to $st$. The remaining triangles can then be added in the above order. Thus a $2$-tree construction may begin with any prescribed root edge.

\begin{definition}\label{def:edge-attachment}
Let $H$ and $J$ be vertex-disjoint graphs with specified edges $xy\in E(H)$ and $st\in E(J)$. An \emph{edge attachment} of $J$ to $H$ along $xy$ identifies $s$ with $x$, $t$ with $y$, and the edge $st$ with $xy$. The common edge is retained once, and all other edges are retained. After identification, the resulting graph is $H\cup J$, with $V(H)\cap V(J)=\{x,y\}$ and $E(H)\cap E(J)=\{xy\}$.
\end{definition}

A nontrivial rooted $2$-tree $T$ attached to an edge $xy$ of $H$ along its root edge is called an \emph{edge decoration} of $xy$. Edge decorations are attached successively along the prescribed edges of $H$. Decorations on the same edge can be combined into a single rooted $2$-tree.

\subsection{Two-terminal graphs and counting parameters}
A \emph{two-terminal graph} $(G,s,t)$ is a graph $G$ with two distinct specified vertices $s$ and $t$, called its \emph{terminals}. The terminals $s$ and $t$ are called the \emph{source} and the \emph{sink}, respectively; these names do not impose restrictions on an orientation. When the terminals are understood, we simply write $G$. Unless stated otherwise, we assume that the underlying graph $G$ is connected.

Following Noble and Royle~\cite[Section~3]{NR}, we introduce three additional notions for two-terminal graphs. A \emph{spanning $2$-forest} of a two-terminal graph $G$ is a spanning subgraph with two components, each of which is a tree, with one containing the source and the other containing the sink. A \emph{very acyclic orientation} of $G$ is an acyclic orientation in which there is no directed path between the two terminals. An \emph{almost totally cyclic orientation} of $G$ is an orientation in which every edge lies on a directed cycle or on a directed path between the two
terminals. The six counting parameters used below are listed in Table~\ref{tab:twoterminal-parameters}. They are unchanged when the two terminals are interchanged.
\begin{table}[htbp]
\centering
\begin{tabular}{@{\hspace{6pt}}l@{\hspace{18pt}}l@{\hspace{6pt}}}
\toprule
Notation & Objects counted\\
\midrule
$\tau(G)$ & Spanning trees\\[2pt]
$\tau_2(G)$ & Spanning $2$-forests separating the terminals\\[2pt]
$\alpha(G)$ & Acyclic orientations\\[2pt]
$\alpha_2(G)$ & Very acyclic orientations\\[2pt]
$\alpha^*(G)$ & Totally cyclic orientations\\[2pt]
$\alpha_2^*(G)$ & Almost totally cyclic orientations\\
\bottomrule
\end{tabular}
\caption{Counting parameters for two-terminal graphs $G$.} \label{tab:twoterminal-parameters}
\end{table}

For distinct vertices $u,v\in V(G)$, let $G/\{u,v\}$ denote the graph obtained by identifying $u$ and $v$, deleting any resulting loops and retaining parallel edges. Then
\begin{equation}\label{spanning_forest_number}
\tau_2(G)=\tau(G/\{s,t\}),
\end{equation}
since identifying $s$ and $t$ gives a bijection between spanning $2$-forests separating the terminals and spanning trees of $G/\{s,t\}$. Connectedness of $G$ implies that $\tau(G)>0$ and $\tau_2(G)>0$. We set
\begin{equation}\label{eq:terminal-r}
r(G)=\frac{\tau(G)}{\tau_2(G)}.
\end{equation}

Regard $G$ as an electrical network with unit resistance on every edge. Then $r(G)$ is the effective conductance between $s$ and $t$ in $G$, and its reciprocal $1/r(G)=\tau_2(G)/\tau(G)$ is the effective resistance, also called the resistance distance; see Kassel, Kenyon and Wu~\cite[Section~2.1, equation~(3)]{KKW}. The Dirichlet principle \cite[Exercise~1.3.11]{DoyleSnell} gives
\begin{equation}\label{eq:r-variational}
r(G)=\min\left\{\sum_{uv\in E(G)}(z_u-z_v)^2:
z:V(G)\to\mathbb R,\ z_s=1,\ z_t=0\right\}.
\end{equation}
Here $z_v=z(v)$, and parallel edges are counted separately in the sum. The following lemma gives a form of Rayleigh monotonicity that allows additional vertices; see D\"orfler and Bullo~\cite[Theorem~III.14(2)]{DB} for the corresponding effective resistance comparison.

\begin{lemma}\label{lem:r-monotonicity}
Let $G$ and $G'$ be connected graphs with the same terminals $s,t$. If $G$ is a subgraph of $G'$, then $r(G')\geq r(G)$.
\end{lemma}
\begin{proof}
For any $z:V(G')\to\mathbb R$ with $z_s=1$ and $z_t=0$, its restriction to $V(G)$ is admissible in \eqref{eq:r-variational}. Since every additional edge contributes a nonnegative term,
\[
\sum_{uv\in E(G')}(z_u-z_v)^2\geq\sum_{uv\in E(G)}(z_u-z_v)^2\geq r(G).
\]
Taking the minimum over $z$ gives $r(G')\geq r(G)$.
\end{proof}

We collect the six counting parameters in Table~\ref{tab:twoterminal-parameters} into the \emph{parameter vector}
\[
\mathbf{v}(G)=\bigl(\tau(G),\tau_2(G),\alpha(G),\alpha_2(G),\alpha^*(G),\alpha_2^*(G)\bigr).
\]
For the single edge $K_2$ and the path $P_3$ on three vertices, each with its ends as terminals, respectively,
$\mathbf{v}(K_2)=(1,1,2,0,0,2)$ and $\mathbf{v}(P_3)=(1,2,4,2,0,2)$.

Take two-terminal graphs $(G_i,s_i,t_i)$, $i=1,2$, with vertex-disjoint underlying graphs. Their \emph{parallel connection} $G_1\paral G_2$ is obtained by identifying $s_1$ with $s_2$ and $t_1$ with $t_2$; the two identified vertices are the new source and sink. Their \emph{series connection} $G_1\ser G_2$ is obtained by identifying $t_1$ with $s_2$, with $s_1$ and $t_2$ as the new source and sink. All edges are retained. A graph obtainable from the two-terminal graph $K_2$ by repeated series and parallel connections is called a \emph{two-terminal series-parallel graph}.

The following composition identities of Noble and Royle~\cite[Lemmas~3.3--3.4]{NR} extend to arbitrary connected two-terminal graphs by the same counting arguments used in their proofs.

\begin{lemma}[\cite{NR}]\label{lem:series-composition}
Let $G_1$ and $G_2$ be connected two-terminal graphs. If $G=G_1\ser G_2$, then
\begin{align*}
&\tau(G)=\tau(G_1)\tau(G_2),\\
&\tau_2(G)=\tau(G_1)\tau_2(G_2)+\tau_2(G_1)\tau(G_2),\\
&\alpha(G)=\alpha(G_1)\alpha(G_2),\\
&\alpha_2(G)=\alpha(G_1)\alpha(G_2)-\frac{\bigl(\alpha(G_1)-\alpha_2(G_1)\bigr)\bigl(\alpha(G_2)-\alpha_2(G_2)\bigr)}{2},\\
&\alpha_2^*(G) =\alpha_2^*(G_1)\alpha_2^*(G_2) -\frac{ \bigl(\alpha_2^*(G_1)-\alpha^*(G_1)\bigr) \bigl(\alpha_2^*(G_2)-\alpha^*(G_2)\bigr) }{2},\\
&\alpha^*(G)=\alpha^*(G_1)\alpha^*(G_2).
\end{align*}
\end{lemma}

\begin{lemma}[\cite{NR}]\label{lem:parallel-composition}
Let $G_1$ and $G_2$ be connected two-terminal graphs. If $G=G_1\paral G_2$, then
\begin{align*}
&\tau(G)=\tau(G_1)\tau_2(G_2)+\tau_2(G_1)\tau(G_2),\\
&\tau_2(G)=\tau_2(G_1)\tau_2(G_2),\\
&\alpha(G) =\alpha(G_1)\alpha(G_2) -\frac{ \bigl(\alpha(G_1)-\alpha_2(G_1)\bigr) \bigl(\alpha(G_2)-\alpha_2(G_2)\bigr) }{2},\\
&\alpha_2(G)=\alpha_2(G_1)\alpha_2(G_2),\\
&\alpha_2^*(G)=\alpha_2^*(G_1)\alpha_2^*(G_2),\\
&\alpha^*(G) =\alpha_2^*(G_1)\alpha_2^*(G_2) -\frac{ \bigl(\alpha_2^*(G_1)-\alpha^*(G_1)\bigr) \bigl(\alpha_2^*(G_2)-\alpha^*(G_2)\bigr) }{2}.
\end{align*}
\end{lemma}

For a connected two-terminal graph $G$, let $\widehat G=G\paral K_2$; equivalently, $\widehat G$ is obtained from $G$ by adding a new edge between its terminals. The identities in Lemma \ref{lem:parallel-composition} and \eqref{eq:terminal-r}  give
\begin{align}
&\tau(\widehat G)=\tau(G)+\tau_2(G)=\tau_2(G)(1+r(G)),\qquad &&\tau_2(\widehat G)=\tau_2(G),\label{eq:restore-tree}\\
&\alpha(\widehat G)=\alpha(G)+\alpha_2(G),\qquad &&\alpha_2(\widehat G)=0,\label{eq:restore-a}\\
&\alpha_2^*(\widehat G)=2\alpha_2^*(G),\qquad &&\alpha^*(\widehat G)=\alpha_2^*(G)+\alpha^*(G).\label{eq:restore-b}
\end{align}

We next record two counting formulas for an edge attachment.
\begin{lemma}\label{lem:substitution}
Let $J$ be a graph with $e=xy\in E(J)$, and let $G_e$ be a connected two-terminal graph with terminals $s,t$, vertex-disjoint from $J$. Attach $\widehat{G_{e}}$ to $J$ along $e$ using its added edge $st$, and put $J'=J\cup\widehat{G_{e}}$. Then
\[
\alpha(J')=\alpha(J)\frac{\alpha(G_e)+\alpha_2(G_e)}2,
\qquad
\alpha^*(J')\geq\alpha^*(J)\alpha_2^*(G_e).
\]
\end{lemma}
\begin{proof}
After identification, we write $x,y$ for the terminals of $G_e$ and $\widehat{G_{e}}$. Fix an acyclic orientation of $J$ and, by interchanging $x$ and $y$ if necessary, suppose that $e$ is directed from $x$ to $y$. By \eqref{eq:restore-a} and reversal symmetry, exactly
\[
\frac{\alpha(G_e)+\alpha_2(G_e)}2
\]
acyclic orientations of $\widehat{G_{e}}$ agree with this direction of $e$. Each such combines with the fixed orientation of $J$ to give an acyclic orientation of $J'$. Indeed, any directed cycle using edges from both graphs $J$ and $\widehat{G_{e}}$  would contain a directed path from $y$ to $x$ in one of them. Together with $e$, this path would form a directed cycle in that graph, a contradiction. Conversely, every acyclic orientation of $J'$ restricts to acyclic orientations of
$J$ and $\widehat{G_{e}}$ agreeing on $e$. Hence
\[
\alpha(J')=\alpha(J)\frac{\alpha(G_e)+\alpha_2(G_e)}2.
\]

Next fix a totally cyclic orientation of $J$ with $e$ directed from $x$ to $y$. By \eqref{eq:restore-b} and reversal symmetry, exactly $\alpha_2^*(G_e)$ almost totally cyclic orientations of $\widehat{G_{e}}$ agree with this direction of $e$. Since $e$ lies on a directed cycle in $J$, deleting $e$ from that cycle gives a directed path from $y$ to $x$.
The edge $e$ itself gives a directed path from $x$ to $y$. Thus every directed path in $\widehat{G_{e}}$ from $x$ to $y$
or from $y$ to $x$ can be completed to a directed cycle using a directed path in $J$ in the opposite direction. Every edge of $\widehat{G_{e}}$ lies on a directed cycle or on such a path, while every edge of $J$ already lies on a directed cycle. Hence the combined orientation of $J'$ is totally cyclic. For each fixed orientation of $J$, these choices give
$\alpha_2^*(G_e)$ distinct totally cyclic orientations of $J'$. Since $J$ has $\alpha^*(J)$ totally cyclic orientations,
we obtain
\[
\alpha^*(J')\geq\alpha^*(J)\alpha_2^*(G_e).\qedhere
\]
\end{proof}

\section{\texorpdfstring{Spanning tree counts and comparisons}{Spanning tree counts and comparisons}}\label{sec:networks}

\begin{lemma}\label{lem:two-complete-tree}
If $J_0=K_p\cup K_q$, where $V(K_p)\cap V(K_q)=S$ and $|S|=a$ with $2\leq a\leq\min\{p,q\}$, then
\begin{equation}\label{eq:clique-tree-ratio}
\tau(J_0) =ap^{p-a-1}q^{q-a-1}(p+q-a)^{a-1}.
\end{equation}
\end{lemma}
\begin{proof}
Suppose first that $a<\min\{p,q\}$. The sets $V(K_p)\setminus S$, $S$ and $V(K_q)\setminus S$ induce complete graphs of orders $p-a$, $a$ and $q-a$, respectively. Any two vertices in consecutive parts are adjacent, and there are no edges between the first and third parts. Thus $J_0$ is a clique extension of a three-vertex path.  Zhou and Bu \cite[Example~6.26]{ZhouBu} give
\[
\begin{aligned}
\tau(J_0) = \frac{p^{p-a-1}}{p-a} \frac{(p+q-a)^{a-1}}{a} \frac{q^{q-a-1}}{q-a} \bigl((p-a)a\bigr)\bigl(a(q-a)\bigr)=ap^{p-a-1}q^{q-a-1}(p+q-a)^{a-1}.
\end{aligned}
\]
The cases $p=a$ or $q=a$ follow from Cayley's formula.
\end{proof}

Chen and Xie~\cite[Lemma~2.5]{ChenXie} proved a related comparison for the addition of a single vertex. The following consequence of Lemma~\ref{lem:r-monotonicity} bounds the spanning tree ratio by replacing the original graph with a complete subgraph containing all shared vertices.

\begin{lemma}\label{lem:attachment-tree-comparison}
Let $H$ be a connected simple graph containing a complete subgraph $K_p$. Suppose that $J$ is a connected graph such that $V(H)\cap V(J)=S\neq\emptyset$ and $S\subseteq V(K_p)$. Put $G=H\cup J$ and $G_p=K_p\cup J$. Then
\begin{equation}\label{eq:attachment-tree-comparison}
\frac{\tau(G)}{\tau(H)}\leq\frac{\tau(G_p)}{\tau(K_p)}.
\end{equation}
\end{lemma}
\begin{proof}
Starting with $K_p$, add the vertices of $V(H)\setminus V(K_p)$ one at a time, each together with one edge of $H$ joining it to a vertex of $K_p$ or to a vertex added earlier. This is possible since $H$ is connected. Then add all remaining edges of $H$. Perform the same additions starting with $G_p=K_p\cup J$. Since $V(H)\cap V(J)\subseteq V(K_p)$, each new vertex lies outside $V(J)$. When added, it has degree one in the graph constructed from $K_p$ and in the graph constructed from $G_p$. Adding a vertex of degree one does not change the number of spanning trees, since every spanning tree must contain its unique incident edge.

When a remaining edge of $H$ is added, let $X$ be the current subgraph of $H$, put $Y=X\cup J$, and let $e=xy$ be the edge to be added at this step. Take $x,y$ as the terminals of $X$ and $Y$. Since $Y$ contains $X$, Lemma \ref{lem:r-monotonicity} gives $r(Y)\geq r(X)>0$. By \eqref{eq:restore-tree} and \eqref{eq:terminal-r},
\begin{equation}\label{eq:attachment-ratio-step}
\begin{aligned}
\frac{\tau(Y+e)}{\tau(X+e)}=\frac{\tau(Y)+\tau_2(Y)}{\tau(X)+\tau_2(X)}=\frac{\tau(Y)}{\tau(X)}
  \frac{1+1/r(Y)}{1+1/r(X)}\leq\frac{\tau(Y)}{\tau(X)}.
\end{aligned}
\end{equation}
Thus the ratio of the spanning tree counts is unchanged during the vertex additions and cannot increase during the edge additions. The two constructions start with $K_p$ and $G_p$ and end with $H$ and $G$, respectively. Therefore
\[
\frac{\tau(G)}{\tau(H)}\leq\frac{\tau(G_p)}{\tau(K_p)}.\qedhere
\]
\end{proof}

\begin{lemma}\label{lem:clique-response}
For $p\geq3$, let $s,t$ be the terminals of $K_p-st$. Then $r(K_p-st)=\frac{p-2}{2}$.
\end{lemma}
\begin{proof}
Put $U=V(K_p)\setminus\{s,t\}$, so $|U|=p-2$. Let $z:V(K_p)\to\mathbb R$ satisfy $z_s=1$ and $z_t=0$. The edges of $K_p-st$ are $sv,tv$ for $v\in U$, together with the edges of $K_p[U]$. Therefore
\[
\begin{aligned}
\sum_{xy\in E(K_p-st)}(z_x-z_y)^2
&=\sum_{v\in U}\bigl((z_s-z_v)^2+(z_t-z_v)^2\bigr)
  +\sum_{xy\in E(K_p[U])}(z_x-z_y)^2\\
&=\sum_{v\in U}\bigl((1-z_v)^2+z_v^2\bigr)
  +\sum_{xy\in E(K_p[U])}(z_x-z_y)^2\\
&=\frac{|U|}{2}
  +2\sum_{v\in U}\left(z_v-\frac12\right)^2
  +\sum_{xy\in E(K_p[U])}(z_x-z_y)^2\\
&\geq\frac{|U|}{2}=\frac{p-2}{2}.
\end{aligned}
\]
Equality holds when $z_v=1/2$ for every $v\in U$. Taking the minimum in \eqref{eq:r-variational} proves the result.
\end{proof}

\section{\texorpdfstring{Rooted $2$-trees and edge decorations}{Rooted 2-trees and edge decorations}}\label{sec:interface}
Let $T$ be a nontrivial rooted $2$-tree with root edge $st$, and put $G=T-st$, with terminals $s,t$. In the present section, we will prove that $G$ is connected and satisfies
\begin{equation}\label{eq:interface}
\bigl(\alpha(G)+\alpha_2(G)\bigr)\alpha_2^*(G) \geq2\tau_2(G)^2\exp\!\left(\frac{4r(G)}{5}\right).
\end{equation}
We prove \eqref{eq:interface} by strong induction on the order of $T$ and use it to derive a sufficient condition under which edge decorations preserve the multiplicative Merino--Welsh inequality. We also recall the known result for nontrivial $2$-trees.

\begin{figure}[htbp]
\centering
\begin{tikzpicture}[x=1cm,y=1cm,
 dot/.style={circle,fill,inner sep=1.7pt},
 box/.style={draw,rounded corners,fill=white,inner sep=5pt}]
 \node[dot,label=left:$s$] (s) at (0,0) {};
 \node[dot,label=above:$w$] (w) at (3,1.2) {};
 \node[dot,label=right:$t$] (t) at (6,0) {};
 \draw (s)--(w)--(t);
 \node[box] at (1.5,.6) {$T_1$};
 \node[box] at (4.5,.6) {$T_2$};
 \draw[dashed] (s)--node[below] {root edge $st$ (deleted)}(t);
 \draw (s) to[bend right=32] node[box,below=3pt] {other branches, if any} (t);
\end{tikzpicture}
\caption{A branch $T_1\ser T_2$ of $T-st$. Only the root edges of $T_1$ and $T_2$ are shown. The lower arc represents the other branches, if any; the
dashed line indicates the deleted edge $st$.} \label{fig:branch}
\end{figure}
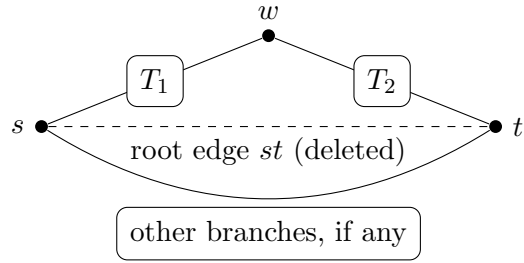

Fix a construction of $T$ starting from $st$, as explained after Definition~\ref{def:2tree}. Each new vertex $w$ added on the edge $st$ is joined to both $s$ and $t$. We start a branch of $T-st$ with the path consisting of the edges $sw$ and $wt$. Whenever a new vertex $v$ is added on an existing edge $xy\neq st$, assign $v$ and its two new edges to the branch containing $xy$. Within the branch corresponding to $w$, the constructions starting from $sw$ and $wt$ yield rooted $2$-trees $T_1$ and $T_2$, respectively; either may be $K_2$. By construction, $V(T_1)\cap V(T_2)=\{w\}$, and there is no edge of $T-st$ between  $V(T_1)\setminus\{w\}$ and $V(T_2)\setminus\{w\}$. With $(s,w)$ and $(w,t)$ as the terminal pairs of $T_1$ and $T_2$, respectively, this branch is their series connection $T_1\ser T_2$, with terminals $s,t$.

Distinct branches meet only at $s,t$ and have no edges between their nonterminal vertex sets. Hence $G$ is a nonempty parallel connection of branches, and is connected. For each branch, $t\notin V(T_1)$ and $s\notin V(T_2)$, so $|V(T_i)|<|V(T)|$ for $i=1,2$. Figure~\ref{fig:branch} illustrates this decomposition.

\begin{lemma}\label{lem:parallel-interface}
Let $G_1$ and $G_2$ be connected two-terminal graphs satisfying \eqref{eq:interface}. Then $G_1\paral G_2$ also satisfies \eqref{eq:interface}.
\end{lemma}
\begin{proof}
Put $G=G_1\paral G_2$. Lemma~\ref{lem:parallel-composition} and \eqref{eq:terminal-r} give
\[
\alpha(G)+\alpha_2(G) =\frac{\bigl(\alpha(G_1)+\alpha_2(G_1)\bigr) \bigl(\alpha(G_2)+\alpha_2(G_2)\bigr)}2,
\]
and
\[
\alpha_2^*(G)=\alpha_2^*(G_1)\alpha_2^*(G_2),\qquad \tau_2(G)=\tau_2(G_1)\tau_2(G_2),\qquad r(G)=r(G_1)+r(G_2).
\]
Substituting these identities and using \eqref{eq:interface} for both $G_1$ and $G_2$, we obtain \eqref{eq:interface} for $G$.
\end{proof}

\begin{lemma}\label{lem:branch}
For $i=1,2$, let $T_i$ be a rooted $2$-tree with root edge $e_i$ and take the ends of $e_i$ as its terminals. For each nontrivial $T_i$, put $G_i=T_i-e_i$ and assume that $G_i$ satisfies \eqref{eq:interface}. Then $T_1\ser T_2$ also satisfies \eqref{eq:interface}.
\end{lemma}
\begin{proof}
Put $H=T_1\ser T_2$. Since $T_1$ and $T_2$ are connected, so is $H$. If $T_1=T_2=K_2$, then $H=P_3$, and its parameter vector gives
\[
\bigl(\alpha(H)+\alpha_2(H)\bigr)\alpha_2^*(H)=12,\qquad \tau_2(H)=2,\qquad r(H)=\frac12.
\]
Thus \eqref{eq:interface} follows from the direct calculation $12>8\exp(2/5)$. Hence assume that at least one of $T_1$ and $T_2$ is nontrivial.

Fix $i\in\{1,2\}$. Suppose first that $T_i$ is nontrivial. Since $G_i=T_i-e_i$, restoring the root edge gives
$T_i=\widehat G_i$. By \eqref{eq:restore-tree}--\eqref{eq:restore-b} and \eqref{eq:terminal-r}, we have
\[
\begin{gathered}
\alpha(T_i)=\alpha(G_i)+\alpha_2(G_i),
\qquad
\alpha_2^*(T_i)=2\alpha_2^*(G_i),\\
\tau_2(T_i)=\tau_2(G_i),
\qquad
r(T_i)=1+r(G_i).
\end{gathered}
\]
Since $G_i$ satisfies \eqref{eq:interface} by hypothesis, we obtain
\[
\begin{aligned}
\alpha(T_i)\alpha_2^*(T_i)
=2\bigl(\alpha(G_i)+\alpha_2(G_i)\bigr)\alpha_2^*(G_i)\geq4\tau_2(G_i)^2
  \exp\!\left(\frac{4r(G_i)}5\right)=4\tau_2(T_i)^2
  \exp\!\left(\frac{4(r(T_i)-1)}5\right).
\end{aligned}
\]
If $T_i=K_2$, then $\alpha(T_i)=\alpha_2^*(T_i)=2$ and $\tau_2(T_i)=r(T_i)=1$. Thus the leftmost and rightmost expressions above both equal $4$. Consequently, for $i=1,2$,
\begin{equation}\label{eq:branch-factor}
\alpha(T_i)\alpha_2^*(T_i)
\geq4\tau_2(T_i)^2
\exp\!\left(\frac{4(r(T_i)-1)}5\right),
\end{equation}
with equality when $T_i=K_2$.

By \eqref{eq:restore-a} and $\alpha_2(K_2)=0$, we have $\alpha_2(T_i)=0$ for $i=1,2$. By Lemma~\ref{lem:series-composition},
\begin{equation}\label{eq:branch-acyclic}
\begin{split}
\alpha(H)+\alpha_2(H)=2\alpha(T_1)\alpha(T_2) -\frac{\alpha(T_1)\alpha(T_2)}2=\frac32\alpha(T_1)\alpha(T_2).
\end{split}
\end{equation}

For each nontrivial $T_i$, \eqref{eq:restore-b} gives $\alpha^*(T_i)=\alpha_2^*(G_i)+\alpha^*(G_i)\geq\alpha_2^*(G_i)=\frac12\alpha_2^*(T_i)$. Since every totally cyclic orientation is almost totally cyclic, $0\leq\alpha_2^*(T_i)-\alpha^*(T_i)\leq\alpha_2^*(T_i)$ for $i=1,2$, and the upper bound improves to $\alpha_2^*(T_i)/2$ if $T_i$ is nontrivial. As at least one $T_i$ is nontrivial,
\[
\begin{aligned}
0\leq\bigl(\alpha_2^*(T_1)-\alpha^*(T_1)\bigr)\bigl(\alpha_2^*(T_2)-\alpha^*(T_2)\bigr)\leq\frac12\alpha_2^*(T_1)\alpha_2^*(T_2).
\end{aligned}
\]
Lemma~\ref{lem:series-composition} then yields
\begin{equation}\label{eq:branch-almost-cyclic}
\begin{split}
\alpha_2^*(H) =\alpha_2^*(T_1)\alpha_2^*(T_2)- \frac{ \bigl(\alpha_2^*(T_1)-\alpha^*(T_1)\bigr) \bigl(\alpha_2^*(T_2)-\alpha^*(T_2)\bigr)
}{2}\geq\frac34\alpha_2^*(T_1)\alpha_2^*(T_2).
\end{split}
\end{equation}

Put $r=r(T_1)+r(T_2)$. Since $r(T_i)=1$ when $T_i=K_2$ and $r(T_i)=1+r(G_i)>1$ otherwise, we have $r>2$. Multiplying \eqref{eq:branch-acyclic} and \eqref{eq:branch-almost-cyclic}, and applying \eqref{eq:branch-factor} to both $T_1$ and $T_2$, we obtain
\begin{equation}\label{eq:branch-orientation-product}
\begin{aligned}
\bigl(\alpha(H)+\alpha_2(H)\bigr)\alpha_2^*(H) &\geq\frac32\cdot\frac34 \bigl(\alpha(T_1)\alpha_2^*(T_1)\bigr)
\bigl(\alpha(T_2)\alpha_2^*(T_2)\bigr)\\
&\geq\frac32\cdot\frac34\cdot4^2 \tau_2(T_1)^2\tau_2(T_2)^2 \exp\!\left(\frac{4(r-2)}5\right)\\
&=18\tau_2(T_1)^2\tau_2(T_2)^2 \exp\!\left(\frac{4(r-2)}5\right).
\end{aligned}
\end{equation}

For the spanning tree counts, Lemma~\ref{lem:series-composition} and \eqref{eq:terminal-r} give
\[
\tau(H) =\tau(T_1)\tau(T_2) =\tau_2(T_1)\tau_2(T_2)r(T_1)r(T_2),
\]
\begin{equation}\label{eq:branch-tree-parameters-1}
\tau_2(H) =\tau(T_1)\tau_2(T_2) +\tau_2(T_1)\tau(T_2)=\tau_2(T_1)\tau_2(T_2) \bigl(r(T_1)+r(T_2)\bigr)=\tau_2(T_1)\tau_2(T_2)r.
\end{equation}
Therefore, by the arithmetic--geometric mean inequality,
\begin{equation}\label{eq:branch-tree-parameters-2}
r(H) =\frac{\tau(H)}{\tau_2(H)} =\frac{r(T_1)r(T_2)}r \leq\frac{\bigl(r(T_1)+r(T_2)\bigr)^2}{4r} =\frac r4.
\end{equation}
Substituting \eqref{eq:branch-tree-parameters-1} and \eqref{eq:branch-tree-parameters-2} into the right-hand side of \eqref{eq:interface} gives
\begin{equation}\label{eq:branch-interface-upper}
\begin{aligned}
2\tau_2(H)^2\exp\!\left(\frac{4r(H)}5\right)= 2\tau_2(T_1)^2\tau_2(T_2)^2r^2 \exp\!\left(\frac{4r(H)}5\right)\leq 2\tau_2(T_1)^2\tau_2(T_2)^2r^2
\exp\!\left(\frac r5\right).
\end{aligned}
\end{equation}

Comparing \eqref{eq:branch-orientation-product} with \eqref{eq:branch-interface-upper}, cancelling the positive factor $\tau_2(T_1)^2\tau_2(T_2)^2$, and dividing by $2\exp(r/5)>0$, it suffices to prove
\begin{equation}\label{eq:two-child-scalar}
9\exp\!\left(\frac{3r-8}5\right)\geq r^2.
\end{equation}

Using $\ln x\leq x-1$ for $x>0$, with $x=3r/10>0$, we obtain
\[
2\ln\frac{3r}{10} \leq2\left(\frac{3r}{10}-1\right) =\frac{3r-10}{5}.
\]
Exponentiating and multiplying by $100/9$ gives
\[
r^2 \leq\frac{100}{9}\exp\!\left(\frac{3r-10}5\right) <9\exp\!\left(\frac{3r-8}5\right),
\]
where the strict inequality follows from the direct calculation $\frac{81}{100}\exp(2/5)>1$. Thus \eqref{eq:two-child-scalar} holds, and \eqref{eq:branch-orientation-product} together with \eqref{eq:branch-interface-upper} proves \eqref{eq:interface} for $H$.
\end{proof}

\begin{theorem}\label{thm:interface}
Let $T$ be a nontrivial rooted $2$-tree with root edge $st$. Then $G=T-st$ with terminals $s$ and $t$ is connected and satisfies \eqref{eq:interface}.
\end{theorem}
\begin{proof}
By the preceding decomposition, $G$ is a nonempty parallel connection of branches. Each branch has the form $T_1\ser T_2$, where $T_1$ and $T_2$ are the two rooted $2$-trees in that branch. Since both are connected, every branch is connected, and hence $G$ is connected.

We prove \eqref{eq:interface} by strong induction on $|V(T)|$. If $|V(T)|=3$, then $G=P_3$, and the case $T_1=T_2=K_2$ in Lemma~\ref{lem:branch} gives \eqref{eq:interface}. Suppose that $|V(T)|>3$ and that the theorem holds for every nontrivial rooted $2$-tree with fewer vertices than $T$. Consider any branch $T_1\ser T_2$ of $G$. For each $i\in\{1,2\}$ such that $T_i$ is nontrivial, we have $|V(T_i)|<|V(T)|$. Thus the induction hypothesis implies that $G_i=T_i-e_i$, where $e_i$ is the root edge of $T_i$, is connected and satisfies \eqref{eq:interface}. Lemma~\ref{lem:branch} therefore gives \eqref{eq:interface} for every branch. Combining the branches successively using Lemma~\ref{lem:parallel-interface} proves the assertion.
\end{proof}

Noble and Royle~\cite[Theorem~4.5]{NR} proved the multiplicative Merino--Welsh inequality for loopless bridgeless series-parallel graphs, which include all nontrivial $2$-trees. Theorem~\ref{thm:interface} gives the following quantitative strengthening for  nontrivial $2$-trees.

\begin{corollary}\label{cor:2tree}
Every nontrivial $2$-tree $T$ satisfies
\[
\alpha(T)\alpha^*(T)\geq\frac{8\exp(6/5)}{25}\tau(T)^2>\tau(T)^2.
\]
\end{corollary}
\begin{proof}
Choose any edge $st$ of $T$ as the root edge, and put $G=T-st$, with terminals $s,t$. By Theorem~\ref{thm:interface}, $G$ is connected. Together with \eqref{eq:restore-tree}--\eqref{eq:restore-b}, this gives
\[
\begin{aligned}
\frac{\alpha(T)\alpha^*(T)}{\tau(T)^2}
&=\frac{(\alpha(G)+\alpha_2(G))(\alpha_2^*(G)+\alpha^*(G))}{\tau_2(G)^2(1+r(G))^2}\\
&\geq\frac{(\alpha(G)+\alpha_2(G))\alpha_2^*(G)}{\tau_2(G)^2(1+r(G))^2}\\
&\geq\frac{2\exp\!\left(4r(G)/5\right)}{(1+r(G))^2}.
\end{aligned}
\]
Let $f(x)=2\exp(4x/5)/(1+x)^2$ for $x>0$. Since
\[
\frac{f'(x)}{f(x)}=\frac{4}{5}-\frac{2}{1+x},
\]
and $f(x)>0$, we have $f'(x)<0$ for $0<x<3/2$ and $f'(x)>0$ for $x>3/2$. Thus $f$ attains its minimum at $x=3/2$. Since $r(G)>0$, it follows that
\[
\frac{\alpha(T)\alpha^*(T)}{\tau(T)^2}\geq\frac{8\exp(6/5)}{25}>1,
\]
where the last inequality follows by direct calculation.
\end{proof}

Let $H$ be a connected graph satisfying the multiplicative Merino--Welsh inequality. The inequality $\alpha(H)\alpha^*(H)\geq\tau(H)^2>0$ implies $\alpha^*(H)>0$, so $H$ has no bridges. For each $e\in E(H)$, the graph $H-e$ is therefore connected, and we take the ends of $e$ as its terminals. We next give a sufficient condition for edge decorations to preserve the multiplicative Merino--Welsh inequality.

\begin{lemma}\label{lem:edgeattach}
Let $H$ be a connected graph satisfying the multiplicative Merino--Welsh inequality, and let $e=st\in E(H)$ with $r(H-e)\geq\frac32$. Let $T_e$ be any nontrivial rooted $2$-tree with root edge $e$ and $V(H)\cap V(T_e)=\{s,t\}$. Then $G=H\cup T_e$ also satisfies the inequality.
\end{lemma}
\begin{proof}
Put $G_e=T_e-e$ with terminals $s,t$. Then $T_e=\widehat{G_{e}}$. By Lemma~\ref{lem:substitution} and Theorem \ref{thm:interface},
\begin{equation}\label{eq:edge-orientation-bound}
\frac{\alpha(G)\alpha^*(G)}{\alpha(H)\alpha^*(H)} \geq \frac{\bigl(\alpha(G_e)+\alpha_2(G_e)\bigr)\alpha_2^*(G_e)}2 \geq
\tau_2(G_e)^2\exp\!\left(\frac{4r(G_e)}5\right).
\end{equation}

Since $H=(H-e)\paral K_2$ and $G=(H-e)\paral\widehat{G_{e}}$, Lemma~\ref{lem:parallel-composition}, \eqref{eq:terminal-r}, and \eqref{eq:restore-tree} yield
\begin{equation}\label{eq:edge-tree-ratio}
\begin{aligned}
\frac{\tau(G)}{\tau(H)}
&=\frac{\tau(H-e)\tau_2(\widehat{G_{e}})
       +\tau_2(H-e)\tau(\widehat{G_{e}})}
       {\tau(H-e)+\tau_2(H-e)}\\
       &=\frac{\tau_2(H-e)\tau_2(G_e)
        \bigl(1+r(H-e)+r(G_e)\bigr)}
       {\tau_2(H-e)\bigl(1+r(H-e)\bigr)}\\
&=\tau_2(G_e)
  \left(1+\frac{r(G_e)}{1+r(H-e)}\right).
\end{aligned}
\end{equation}
Using $\ln(1+x)\leq x$ for $x\geq0$, with $x=r(G_e)/(1+r(H-e))$, and $r(H-e)\geq3/2$, we obtain
\[
2\ln\left(1+\frac{r(G_e)}{1+r(H-e)}\right) \leq\frac{2r(G_e)}{1+r(H-e)} \leq\frac45r(G_e).
\]
Exponentiating and using \eqref{eq:edge-tree-ratio} gives
\begin{equation}\label{eq:edge-tree-bound}
\left(\frac{\tau(G)}{\tau(H)}\right)^2 \leq \tau_2(G_e)^2\exp\!\left(\frac{4r(G_e)}5\right).
\end{equation}

By \eqref{eq:edge-orientation-bound}, \eqref{eq:edge-tree-bound}, and $\alpha(H)\alpha^*(H)\geq\tau(H)^2$, we conclude that
\[
\alpha(G)\alpha^*(G) \geq \alpha(H)\alpha^*(H) \left(\frac{\tau(G)}{\tau(H)}\right)^2 \geq\tau(G)^2.\qedhere
\]
\end{proof}

\begin{corollary}\label{cor:large-edge}
Let $H$ be a connected graph satisfying the multiplicative Merino--Welsh inequality. If both ends of an edge $e$ lie in a clique of $H$ with at least five vertices, then attaching any nontrivial rooted $2$-tree along $e$ to $H$ preserves the inequality.
\end{corollary}
\begin{proof}
Let $e=st$ belong to a complete subgraph $K_p$ of $H$, where $p\geq5$. With $s,t$ as terminals, $K_p-e$ is a subgraph of $H-e$. By Lemma~\ref{lem:r-monotonicity}, $r(H-e)\geq r(K_p-e)$. Lemma~\ref{lem:clique-response} gives \[r(K_p-e)=\frac{p-2}{2}\geq\frac32.\]
Thus $r(H-e)\geq3/2$, and the conclusion follows from Lemma~\ref{lem:edgeattach}.
\end{proof}

\section{\texorpdfstring{Complete graph attachments and decorated $K_4$ graphs}{Complete graph attachments and decorated K4
graphs}}\label{sec:cliques}

In this section, we compare orientation and spanning tree counts when a complete graph is attached and edge decorations are then added. These comparisons give sufficient conditions for preserving the multiplicative Merino--Welsh inequality. Once attaching $K_q$ with $q\geq5$ is shown to preserve the inequality, Corollary~\ref{cor:large-edge} allows edge decorations to be added successively along its edges. For $K_4$, we treat attachments along an edge and along a triangle separately.

\subsection{Complete graph attachments}
Let $S$ be a clique of $H$ and let $q\geq |S|$. Choose a copy of $K_q$ such that $V(H)\cap V(K_q)=S$. We call the operation of forming $H\cup K_q$ a \emph{complete graph attachment} of $K_q$ to $H$ along $S$.

\begin{lemma}\label{lem:clique-gains}
Let $H$ be a connected simple graph, and let $S$ be a clique of order $a\geq2$ contained in a clique of $H$ of order $p$. Let $K_q$ be a complete graph with $q>a$ and $V(H)\cap V(K_q)=S$, and put $G=H\cup K_q$. Then
\begin{align}
&\frac{\alpha(G)}{\alpha(H)} =\frac{q!}{a!},\label{eq:clique-a}\\
&\frac{\alpha^*(G)}{\alpha^*(H)} \geq\prod_{d=a}^{q-1}(2^d-2) \quad\text{if~}\alpha^*(H)>0,\label{eq:clique-c}\\
&\frac{\tau(G)}{\tau(H)} \leq aq^{q-a-1} \left(1+\frac{q-a}{p}\right)^{a-1}. \label{eq:clique-t}
\end{align}
\end{lemma}
\begin{proof}
Starting from $H$, add the vertices of $V(K_q)\setminus S$ in a fixed order to obtain $G$. At each step, let $J$ be the graph already constructed and $v$ be the newly added vertex, and let $J'$ be the resulting graph. The neighbors of $v$ in $J'$ are precisely the vertices of $S$ and the vertices of $V(K_q)\setminus S$ already added. They form a clique of $J$ of order $d$, where $d$ takes the values $a,a+1,\ldots,q-1$.

For \eqref{eq:clique-a}, fix an acyclic orientation of $J$. Choose an ordering of the vertices of $J$ such that, whenever two vertices are adjacent, their edge is directed from the earlier vertex to the later one. Such an ordering is obtained by repeatedly choosing and deleting a vertex with no incoming edges, which always exists in a nonempty acyclic digraph~\cite[Section~1]{Stanley}. Since the neighbors of $v$  form a clique in $J$, they have a unique order $u_1,u_2,\ldots,u_d$ in which every edge $u_iu_j$ with $i<j$ is directed from $u_i$ to $u_j$. For each $k\in\{0,1,\ldots,d\}$, insert $v$ into the chosen ordering of $V(J)$: place it first if $k=0$, and immediately
after $u_k$ if $k\geq1$. Keep all edge directions in $J$ unchanged, and direct $vu_i$ towards $v$ if $i\leq k$ and away from $v$ if $i>k$. Every edge of $J'$ is then directed from an earlier vertex to a later one in the resulting ordering, so the orientation is acyclic. Thus these $d+1$ choices give distinct acyclic orientations of $J'$.

Conversely, consider any acyclic orientation of $J'$ in which the directions of all edges of $J$ are unchanged. If $vu_j$ is directed towards $v$, then every edge $vu_i$ with $i<j$ must also be directed towards $v$. Otherwise, the edges $vu_i$, $u_iu_j$ and $u_jv$ would form a directed cycle. Hence exactly the first $k$ neighbors have their edges directed towards $v$, for a unique $k\in\{0,1,\ldots,d\}$. Thus every acyclic orientation of $J$ has exactly $d+1$ acyclic extensions, giving $\alpha(J')=(d+1)\alpha(J)$. Applying this identity at each step, we obtain
\[\frac{\alpha(G)}{\alpha(H)}=\prod_{d=a}^{q-1}(d+1)=\frac{q!}{a!}.\]

For \eqref{eq:clique-c}, suppose that $\alpha^*(H)>0$. At each step, fix a totally cyclic orientation of $J$. Since $J$ is connected, this orientation is strongly connected. For each such orientation, the $d$ newly added edges have $2^d$ possible orientations. Excluding the two choices in which all new edges are directed towards $v$ or all are directed
away from $v$ leaves $2^d-2$ choices, each giving $v$ at least one incoming and one outgoing edge. An incoming edge lets every vertex of $J$ reach $v$, and an outgoing edge lets $v$ reach every vertex of $J$, since the orientation of $J$ is strongly connected. Thus the resulting orientation of $J'$ is strongly connected and hence totally cyclic. Different choices of the new edge directions give different orientations of $J'$, and orientations obtained from different orientations of $J$ are also distinct. Consequently, $\alpha^*(J')\geq(2^d-2)\alpha^*(J)$. Applying this inequality at each step and dividing by $\alpha^*(H)$, we obtain
\[\frac{\alpha^*(G)}{\alpha^*(H)}\geq\prod_{d=a}^{q-1}(2^d-2).\]

For \eqref{eq:clique-t}, choose a complete subgraph $K_p$ of $H$ containing $S$, and put $J_0=K_p\cup K_q$. By Lemma~\ref{lem:attachment-tree-comparison} and Lemma~\ref{lem:two-complete-tree},
\[
\frac{\tau(G)}{\tau(H)}
\leq\frac{\tau(J_0)}{\tau(K_p)}
=aq^{q-a-1}\left(1+\frac{q-a}{p}\right)^{a-1}.\qedhere
\]
\end{proof}

\begin{lemma}\label{lem:highdegree}
Let $H$ be a connected simple graph satisfying the multiplicative Merino--Welsh inequality. Let $G$ be obtained from $H$ by adding a new vertex adjacent precisely to the vertices of a clique of $H$ of order $d\geq4$. Then $G$ also satisfies this inequality.
\end{lemma}
\begin{proof}
Write $G=H\cup K_{d+1}$, where $K_{d+1}$ is the complete subgraph induced by the new vertex and its neighbors. Applying Lemma~\ref{lem:clique-gains} with $a=p=d$ and $q=d+1$, we obtain
\begin{equation}\label{eq:highdegree-gains}
\frac{\alpha(G)\alpha^*(G)}{\alpha(H)\alpha^*(H)}
\geq(d+1)(2^d-2),
\qquad
\frac{\tau(G)}{\tau(H)}
\leq d\left(1+\frac1d\right)^{d-1}.
\end{equation}
Since $H$ satisfies the multiplicative Merino--Welsh inequality, \eqref{eq:highdegree-gains} shows that it suffices to prove
\[
\left[d\left(1+\frac1d\right)^{d-1}\right]^2
\leq(d+1)(2^d-2).
\]
Dividing both sides by $d+1=d(1+1/d)$ gives the equivalent inequality
\begin{equation}\label{eq:highdegree-scalar}
d\left(1+\frac1d\right)^{2d-3}\leq2^d-2.
\end{equation}

We prove the inequality  \eqref{eq:highdegree-scalar} by induction on $d$. For $d=4$, it follows from $3125/256<14$. Suppose that it holds for some $d\geq4$. By the induction hypothesis, we obtain
\begin{align*}
(d+1)\left(1+\frac1{d+1}\right)^{2d-1} &<(d+1)\left(1+\frac1d\right)^{2d-1}\\
&=\left(1+\frac1d\right)^3d\left(1+\frac1d\right)^{2d-3}\\
&\leq\frac{125}{64}(2^d-2)<2^{d+1}-2.
\end{align*}
This completes the induction and proves \eqref{eq:highdegree-scalar}.
\end{proof}

\begin{lemma}\label{lem:largeclique}
Let $H$ be a connected simple graph satisfying the multiplicative Merino--Welsh inequality. Let $S$ be a clique  of order $a\geq2$, properly contained in a clique of $H$ of order $p$. Let $G=H\cup K_q$, where $q>a$, $q\geq5$ and $V(H)\cap V(K_q)=S$. Then $G$ also satisfies this inequality.
\end{lemma}
\begin{proof}
Note that $p\geq a+1$. Next we consider the following three cases.
\vskip 0.25cm\noindent
\textbf{Case 1:} $a\geq4$. Starting from $H$, add the vertices of $V(K_q)\setminus S$ one at a time. At each step, the new vertex is adjacent precisely to the vertices of $S$ and the new vertices already added. These vertices form a clique of order at least four, so Lemma~\ref{lem:highdegree} applies at every step.

\vskip 0.25cm\noindent
\textbf{Case 2:} $a=3$.
Choose a complete subgraph $K_4$ of $K_q$ containing $S$, and put $H'=H\cup K_4$. Since $p\geq4$, $S$ is contained in a clique of $H$ of order four. Applying Lemma~\ref{lem:clique-gains} with $a=3$, $p=4$ and $q=4$, we obtain
\[
\begin{aligned}
\frac{\tau(H')}{\tau(H)}
&\leq3\left(1+\frac14\right)^2
=\frac{75}{16},\\
\frac{\alpha(H')\alpha^*(H')}{\alpha(H)\alpha^*(H)}
&\geq4(2^3-2)
=24>\left(\frac{75}{16}\right)^2.
\end{aligned}
\]

\vskip 0.25cm\noindent
\textbf{Case 3:} $a=2$.
Choose a complete subgraph $K_5$ of $K_q$ containing $S$, and put $H'=H\cup K_5$. Since $p\geq3$, the edge with ends in $S$ belongs to a triangle of $H$. Applying Lemma~\ref{lem:clique-gains} with $a=2$, $p=3$ and $q=5$, we obtain
\[
\begin{aligned}
\frac{\tau(H')}{\tau(H)}
&\leq2\cdot5^2\left(1+\frac33\right)
=100,\\
\frac{\alpha(H')\alpha^*(H')}{\alpha(H)\alpha^*(H)}
&\geq\frac{5!}{2!}\prod_{d=2}^{4}(2^d-2)
=10080>100^2.
\end{aligned}
\]

In Cases 2 and 3, these bounds and the assumed inequality for $H$ show that $H'$ also satisfies the multiplicative Merino--Welsh inequality. Add any remaining new vertices successively. Each has a neighborhood consisting of $S$ and the new vertices already added, which form a clique of order at least four. Applying Lemma~\ref{lem:highdegree} at each step completes the proof.
\end{proof}

\begin{lemma}[\cite{CM}]\label{lem:complete}
Every complete graph $K_q$ with $q\geq3$ satisfies the multiplicative Merino--Welsh inequality.
\end{lemma}
\begin{proof}
For $q\geq4$, the assertion follows from Conde and Merino~\cite[Lemma~3.4]{CM} by taking $n=q$ and $m=0$. For $q=3$, we have $\alpha(K_3)\alpha^*(K_3)=6\cdot2>3^2=\tau(K_3)^2$.
\end{proof}

\subsection{\texorpdfstring{Decorated $K_4$ graphs}{Decorated K4 graphs}}

We first attach an undecorated $K_4$ and then add the edge decorations successively. The following lemma controls the change in $\alpha(X)\alpha^*(X)/\tau(X)^2$ at each step.

\begin{lemma}\label{lem:decoration-ratio}
Let $X$ be a connected simple graph with $\alpha^*(X)>0$. Let $e=st\in E(X)$, and let $T_e$ be any nontrivial rooted $2$-tree with root edge $e$ and $V(X)\cap V(T_e)=\{s,t\}$. Put $Y=X\cup T_e$. Taking $s,t$ as the terminals of $X-e$, we have
\begin{align}
\frac{\alpha(Y)\alpha^*(Y)}{\tau(Y)^2} &\geq\frac{19}{20}\cdot \frac{\alpha(X)\alpha^*(X)}{\tau(X)^2} \quad\text{if }r(X-e)\geq1,
\label{eq:decoration-ratio-one}\\
\frac{\alpha(Y)\alpha^*(Y)}{\tau(Y)^2} &\geq\frac{35}{36}\cdot \frac{\alpha(X)\alpha^*(X)}{\tau(X)^2} \quad\text{if }r(X-e)\geq\frac{21}{19}.
\label{eq:decoration-ratio-outer}
\end{align}
\end{lemma}
\begin{proof}
Since $\alpha^*(X)>0$, the graph $X$ has no bridges, so $X-e$ is connected. Put $G_e=T_e-e$, with terminals $s,t$. By Theorem~\ref{thm:interface}, $G_e$ is connected and satisfies \eqref{eq:interface}. In a construction of $T_e$ starting from $e$, the first added vertex is adjacent to both $s$ and $t$. Thus $G_e$ contains a path $P_3$ with ends $s,t$, and Lemma~\ref{lem:r-monotonicity} gives
\begin{equation}\label{eq:nontrivial-decoration-r}
r(G_e)\geq r(P_3)=\frac12.
\end{equation}

By Lemma~\ref{lem:substitution} and \eqref{eq:interface},
\begin{equation}\label{eq:decoration-orientation-ratio}
\frac{\alpha(Y)\alpha^*(Y)}{\alpha(X)\alpha^*(X)}
\geq\frac{\bigl(\alpha(G_e)+\alpha_2(G_e)\bigr)\alpha_2^*(G_e)}2
\geq\tau_2(G_e)^2\exp\!\left(\frac{4r(G_e)}5\right).
\end{equation}
Since $X=(X-e)\paral K_2$ and $Y=(X-e)\paral T_e$, Lemma~\ref{lem:parallel-composition}, \eqref{eq:terminal-r} and \eqref{eq:restore-tree} yield
\begin{equation}\label{eq:decoration-tree-ratio}
\begin{aligned}
\frac{\tau(Y)}{\tau(X)}
&=\frac{\tau(X-e)\tau_2(T_e)+\tau_2(X-e)\tau(T_e)}
        {\tau(X-e)+\tau_2(X-e)}\\
&=\frac{\tau_2(X-e)\tau_2(G_e)
        \bigl(1+r(X-e)+r(G_e)\bigr)}
        {\tau_2(X-e)\bigl(1+r(X-e)\bigr)}\\
&=\tau_2(G_e)\left(1+\frac{r(G_e)}{1+r(X-e)}\right).
\end{aligned}
\end{equation}

Put $x=r(G_e)$, so $x\geq1/2$ by \eqref{eq:nontrivial-decoration-r}. For any fixed $\lambda$ satisfying $1\leq\lambda\leq r(X-e)$, define
\[
f_\lambda(z)=\frac{\exp(4z/5)}{\left(1+z/(1+\lambda)\right)^2}
\qquad(z\geq1/2).
\]
Combining \eqref{eq:decoration-orientation-ratio} and \eqref{eq:decoration-tree-ratio}, and using $\lambda\leq r(X-e)$, we obtain
\begin{equation}\label{eq:decoration-ratio-exponential}
\begin{aligned}
\frac{\alpha(Y)\alpha^*(Y)/\tau(Y)^2}
     {\alpha(X)\alpha^*(X)/\tau(X)^2}\geq
\frac{\exp(4x/5)}
     {\left(1+x/(1+r(X-e))\right)^2}\geq f_\lambda(x).
\end{aligned}
\end{equation}

For $z\geq1/2$ and $\lambda\geq 1$, we have $1+\lambda+z\geq5/2$. Hence
\[
f_\lambda'(z)
=f_\lambda(z)\left(\frac45-\frac{2}{1+\lambda+z}\right)\geq0.
\]
Thus $f_\lambda$ is nondecreasing on $[1/2,\infty)$. Since $x\geq1/2$, it follows that
\[
f_\lambda(x)\geq f_\lambda(1/2)
=\exp\!\left(\frac25\right)
 \left(\frac{2(1+\lambda)}{3+2\lambda}\right)^2>\frac{149}{100} \left(\frac{2(1+\lambda)}{3+2\lambda}\right)^2.
\]
Taking $\lambda=1$ when $r(X-e)\geq1$ and $\lambda=21/19$ when $r(X-e)\geq21/19$, we obtain
\begin{align*}
f_1(x)
&>\frac{149}{100}\left(\frac45\right)^2
=\frac{596}{625}>\frac{19}{20},\\
f_{21/19}(x)
&>\frac{149}{100}\left(\frac{80}{99}\right)^2
=\frac{9536}{9801}>\frac{35}{36}.
\end{align*}
Together with \eqref{eq:decoration-ratio-exponential}, these inequalities prove \eqref{eq:decoration-ratio-one} and \eqref{eq:decoration-ratio-outer}, respectively.
\end{proof}

\begin{remark}\label{rem:decoration-positivity}
Under the hypotheses of Lemma~\ref{lem:decoration-ratio}, $\alpha^*(Y)>0$. Thus the positivity assumption is preserved when edge decorations are added successively.
\end{remark}

We next consider an attachment along an edge. For the small graphs below, we compute the orientation counts by hand and the spanning tree counts using the Matrix--Tree Theorem.

\begin{proposition}\label{prop:K4edge}
Let $H$ be a connected simple graph satisfying the multiplicative Merino--Welsh inequality, and let $uv\in E(H)$ lie in a triangle of $H$. Let $G$ be obtained by attaching a complete graph $K_4$ to $H$ along $uv$ and then attaching an edge decoration along each edge in a subset of $E(K_4)\setminus\{uv\}$. Then $G$ also satisfies the multiplicative Merino--Welsh inequality.
\end{proposition}

\begin{proof}
Let $J=H\cup K_4$ be the graph before the decorations are attached, and write $V(K_4)=\{u,v,x,y\}$. Lemma~\ref{lem:clique-gains}, with $a=2$, $p=3$ and $q=4$, gives
\[
\frac{\alpha(J)}{\alpha(H)}=12,
\qquad
\frac{\tau(J)}{\tau(H)}\leq\frac{40}{3}.
\]
Put $D=K_4-uv$, with terminals $u,v$. A direct count gives $\alpha_2^*(D)=18$. Since $J=H\cup K_4=H\cup\widehat D$ and $\alpha^*(H)>0$, Lemma~\ref{lem:substitution} yields
\[
\frac{\alpha^*(J)}{\alpha^*(H)}\geq18.
\]
Consequently,
\begin{equation}\label{eq:K4edge-initial-ratio}
\frac{\alpha(J)\alpha^*(J)}{\tau(J)^2}
\geq\frac{216}{(40/3)^2}
\frac{\alpha(H)\alpha^*(H)}{\tau(H)^2}
=\frac{243}{200}
\frac{\alpha(H)\alpha^*(H)}{\tau(H)^2}.
\end{equation}

Choose a triangle $H[\{u,v,w\}]$ containing $uv$ and put $J_0=H[\{u,v,w\}]\cup K_4$. Thus
\[
V(J_0)=\{u,v,x,y,w\},
\qquad
E(J_0)=\{uv,ux,vx,uy,vy,xy,uw,vw\}.
\]
Equation~\eqref{eq:clique-tree-ratio}, with $p=3$, $q=4$ and $a=2$, gives $\tau(J_0)=40$. Taking $(u,x)$ and $(x,y)$ as the terminal pairs of $J_0-ux$ and $J_0-xy$, respectively, \eqref{spanning_forest_number} and a direct calculation give
\[
\tau_2(J_0-ux)
=\tau\bigl((J_0-ux)/\{u,x\}\bigr)
=19,\qquad
\tau_2(J_0-xy)
=\tau\bigl((J_0-xy)/\{x,y\}\bigr)
=20.
\]
For each $e\in\{ux,vx,uy,vy,xy\}$, take its ends as the terminals of $J_0-e$. By \eqref{eq:restore-tree}, $\tau(J_0-e)=\tau(J_0)-\tau_2(J_0-e)$. By \eqref{eq:terminal-r} and symmetry among the edges $ux,vx,uy,vy$, we obtain
\begin{equation}\label{eq:K4edge-deleted-r}
r(J_0-e)=\frac{40-19}{19}=\frac{21}{19}
\quad(e\in\{ux,vx,uy,vy\}),
\qquad
r(J_0-xy)=\frac{40-20}{20}=1.
\end{equation}

Now attach the decorations successively. At each step, let $X$ be the current graph and $e$ the edge to be decorated, with the ends of $e$ as terminals of $X-e$. Since $\alpha^*(J)>0$ and this positivity is preserved by Remark~\ref{rem:decoration-positivity}, we have $\alpha^*(X)>0$, so $X-e$ is connected. The graph $X-e$ contains $J_0-e$, so Lemma~\ref{lem:r-monotonicity} and \eqref{eq:K4edge-deleted-r} give $r(X-e)\geq21/19$ if $e\in\{ux,vx,uy,vy\}$,
and $r(X-e)\geq1$ if $e=xy$. Thus \eqref{eq:decoration-ratio-outer} applies at most four times and \eqref{eq:decoration-ratio-one} at most once. Combining these bounds with \eqref{eq:K4edge-initial-ratio} and the assumed inequality for $H$, we obtain
\[
\frac{\alpha(G)\alpha^*(G)}{\tau(G)^2}\geq\frac{243}{200}
\left(\frac{35}{36}\right)^4\frac{19}{20}
\frac{\alpha(H)\alpha^*(H)}{\tau(H)^2}\geq\frac{243}{200}
\left(\frac{35}{36}\right)^4\frac{19}{20}
=\frac{228095}{221184}>1.\qedhere
\]
\end{proof}

We next consider an attachment along a triangle.

\begin{proposition}\label{prop:K4triangle}
Let $H$ be a connected simple graph satisfying the multiplicative Merino--Welsh inequality, and let $S=\{s_1,s_2,s_3\}$ be a clique contained in another clique of $H$ of order at least four. Let $G$ be obtained by attaching a complete graph $K_4$ to $H$ along $S$ and then attaching an edge decoration along each edge in a subset of $E(K_4)\setminus E(K_4[S])$.
Then $G$ also satisfies the multiplicative Merino--Welsh inequality.
\end{proposition}
\begin{proof}
Let $J=H\cup K_4$ be the graph before the decorations are attached, and write $V(K_4)=S\cup\{v\}$. Lemma~\ref{lem:clique-gains}, with $a=3$, $p=4$ and $q=4$, gives
\[
\frac{\alpha(J)\alpha^*(J)}{\alpha(H)\alpha^*(H)}
\geq4(2^3-2)=24,
\qquad
\frac{\tau(J)}{\tau(H)}\leq\frac{75}{16}.
\]
Hence
\begin{equation}\label{eq:K4triangle-initial-ratio}
\frac{\alpha(J)\alpha^*(J)}{\tau(J)^2}
\geq\frac{24}{(75/16)^2}
\frac{\alpha(H)\alpha^*(H)}{\tau(H)^2}
=\frac{6144}{5625}
\frac{\alpha(H)\alpha^*(H)}{\tau(H)^2}.
\end{equation}

Choose $w\in V(H)\setminus S$ such that $H[S\cup\{w\}]$ is complete, and put $J_0=H[S\cup\{w\}]\cup K_4$.
Equation~\eqref{eq:clique-tree-ratio}, with $p=q=4$ and $a=3$, gives $\tau(J_0)=75$. Taking $v,s_1$ as the terminals of $J_0-vs_1$, \eqref{spanning_forest_number} and a direct calculation give $\tau_2(J_0-vs_1)=\tau\bigl((J_0-vs_1)/\{v,s_1\}\bigr)=35$. By \eqref{eq:restore-tree}, \eqref{eq:terminal-r} and symmetry among the edges $vs_1,vs_2,vs_3$, we obtain
\begin{equation}\label{eq:K4triangle-deleted-r}
r(J_0-vs_i)=\frac{75-35}{35}
=\frac87>\frac{21}{19}
\qquad(i=1,2,3).
\end{equation}

Now attach the decorations successively. At the step when $vs_i$ is decorated, let $X$ be the current graph. By \eqref{eq:K4triangle-initial-ratio} and Remark~\ref{rem:decoration-positivity}, we have $\alpha^*(X)>0$,
so $X-vs_i$ is connected. We take $v$ and $s_i$ as the terminals of $X-vs_i$. Since $J_0-vs_i\subseteq X-vs_i$, Lemma~\ref{lem:r-monotonicity} and \eqref{eq:K4triangle-deleted-r} give $r(X-vs_i)\geq8/7>21/19$. Thus \eqref{eq:decoration-ratio-outer} applies at most three times. Combining these bounds with
\eqref{eq:K4triangle-initial-ratio} and the assumed inequality for $H$, we obtain
\[
\frac{\alpha(G)\alpha^*(G)}{\tau(G)^2}
\geq\frac{6144}{5625}
\left(\frac{35}{36}\right)^3
\frac{\alpha(H)\alpha^*(H)}{\tau(H)^2}\geq\frac{6144}{5625}
\left(\frac{35}{36}\right)^3
=\frac{10976}{10935}>1.\qedhere
\]
\end{proof}

The following lemma combines the preceding results for complete graph attachments with edge decorations.

\begin{lemma}\label{lem:leafattach}
Let $H$ be a  connected simple graph satisfying the multiplicative Merino--Welsh inequality. Let $S$ be a clique of order $a\geq2$ properly contained in a  clique $P$ of $H$. Let $G$ be obtained by attaching a complete graph $K_q$ to $H$ along $S$, where $q>a$ and $q\geq4$, and then attaching an edge decoration along each edge
in a subset of $E(K_q)\setminus E(K_q[S])$. Then $G$ also satisfies the multiplicative Merino--Welsh inequality.
\end{lemma}

\begin{proof}
If $q\geq5$, Lemma~\ref{lem:largeclique} shows that $H\cup K_q$ satisfies the multiplicative Merino--Welsh inequality. Every edge to be decorated belongs to the attached $K_q$, which remains a subgraph throughout the construction. Hence Corollary~\ref{cor:large-edge} applies successively.

If $q=4$, then $a\in\{2,3\}$. Since $S\subsetneq P$, we have $|P|\geq a+1$. Thus Propositions~\ref{prop:K4edge} and~\ref{prop:K4triangle} apply when $a=2$ and $a=3$, respectively. This completes the proof.
\end{proof}

\section{Proof of the main theorem}\label{sec:mainproof}

In this section, all graphs are simple. For a $2$-connected chordal graph $G$, put
\[
\ell(G)=\bigl|\{C\in\mathcal C(G):|C|\geq4\}\bigr|.
\]
Every maximal clique of $G$ has order at least three. Indeed, if $G$ has more than one maximal clique, each maximal clique properly contains its intersection with a neighboring clique, which has at least two vertices by Lemma~\ref{lem:separators}. Otherwise, $G$ is complete and has at least three vertices by $2$-connectedness. Thus every maximal clique not counted by $\ell(G)$ is a triangle. We first establish the structural decomposition needed for induction on $\ell(G)$.

\begin{lemma}\label{lem:leafdecomposition}
Let $G$ be a noncomplete $2$-connected chordal graph with $\ell(G)\geq1$. Then there exist a $2$-connected chordal graph $H$ with $\ell(H)=\ell(G)-1$ and a clique $S$ of $H$ with $|S|\geq2$, properly contained in a maximal clique of $H$. The graph $G$ is obtained by attaching a complete graph $K_q$ to $H$ along $S$, where $q\geq4$ and $q>|S|$, and then attaching an edge decoration along each edge in a subset of $E(K_q)\setminus E(K_q[S])$.
\end{lemma}
\begin{proof}
Root a clique tree \(\mathcal T\) of \(G\) so that at least one maximal clique of order at least four is distinct from the root.  Let $C$ be a maximal clique of order at least four at maximum distance from the root. Then $C$ has a parent $P$. Put $S=C\cap P$. By Lemma~\ref{lem:separators}, $|S|\geq2$, $S\subsetneq C$ and $S\subsetneq P$. By the choice of $C$, every proper descendant of $C$ corresponds to a maximal clique of order three.

For each child $Q$ of $C$, let $T_Q$ be the union of the triangle subgraphs corresponding to $Q$ and its descendants.
Lemma~\ref{lem:separators} gives $|Q\cap C|=2$; let $e_Q$ be the unique edge of $G[Q\cap C]$. By Lemma~\ref{lem:clique-tree-cut}, \(T_Q\) meets the union of the clique subgraphs corresponding to nodes outside the subtree rooted at \(Q\) precisely in the edge \(e_Q\). The construction following Definition~\ref{def:2tree} shows that
$T_Q$ is a nontrivial rooted $2$-tree with root edge $e_Q$. By Lemma \ref{lem:clique-tree-cut}, for distinct children $Q$ and $Q'$ of $C$, the sets $V(T_Q)\setminus C$ and $V(T_{Q'})\setminus C$ are disjoint and no edge of $G$ joins them.

Let $\mathcal T_P$ be the component of $\mathcal T-CP$ containing $P$, and let $H_P$ be the union of the complete subgraphs corresponding to its nodes. Define
\[
H=H_P\cup
\bigcup_{\substack{Q\text{ a child of }C\\Q\cap C\subseteq S}}T_Q.
\]
By Lemma~\ref{lem:clique-tree-cut} and the construction of $H$, $V(H_P)\cap C=V(H)\cap C=S$.

Add the clique subgraphs corresponding to the nodes of \(\mathcal T_P\) in an order in which each parent precedes its children. By Lemma 2.2, each clique after the first meets the preceding union precisely in its intersection with its parent. By Lemma 2.1, this intersection is a complete subgraph of order at least two. Since every clique subgraph has order at least three, repeated application of Lemma 2.3 shows that \(H_P\) is chordal and \(2\)-connected. Applying the same lemma to the selected rooted $2$-trees, whose root edges lie in $G[S]$, shows that $H$ is chordal and $2$-connected. Since $G[P]\subseteq H\subseteq G$, $P$ remains a maximal clique of $H$. Thus $S$ is properly contained in a maximal clique of $H$. Every clique of $H$ of order at least four is contained in $H_P$.
Note that any maximal clique of $H_P$ of order at least four that is not maximal in $G$ must be contained in $C$, hence in $S\subsetneq P$, contradicting its maximality. Thus the maximal cliques of $H_P$ of order at least four are exactly those of $G$ of order at least four other than $C$. Hence $\ell(H)=\ell(H_P)=\ell(G)-1$.

Put $q=|C|$, so $q\geq4$ and $q>|S|$. Attach $G[C]\cong K_q$ to $H$ along $S$. Each $T_Q$ not included in $H$ meets $H\cup G[C]$ exactly in its root edge and its ends, with at least one end outside $S$. Attach these rooted $2$-trees successively as edge decorations to recover $G$.
\end{proof}

\begin{theorem}\label{thm:2connected}
Every $2$-connected chordal graph satisfies the multiplicative Merino--Welsh inequality.
\end{theorem}

\begin{proof}
We proceed by induction on $\ell(G)$. If $\ell(G)=0$, all maximal cliques of $G$ are triangles, so $G$ is a nontrivial $2$-tree by the characterization following Definition~\ref{def:2tree}. Corollary~\ref{cor:2tree} gives the multiplicative Merino--Welsh inequality for $G$.

Now let $\ell(G)=m\geq1$ and assume that every $2$-connected chordal graph with exactly $m-1$ maximal cliques of order at least four satisfies the inequality. If $G$ is complete, apply Lemma~\ref{lem:complete}. Otherwise, Lemma~\ref{lem:leafdecomposition} gives a $2$-connected chordal graph $H$ with $\ell(H)=m-1$ such that $G$ is obtained
from $H$ by attaching a complete graph and then edge decorations as described in that lemma. By the induction hypothesis, $H$ satisfies the multiplicative Merino--Welsh inequality. Applying Lemma~\ref{lem:leafattach} then gives the multiplicative Merino--Welsh inequality for $G$.
\end{proof}

Now we extend the result to connected bridgeless chordal graphs using block decomposition. A \emph{block} is a maximal $2$-connected subgraph, a bridge together with its ends, or an isolated vertex. Distinct blocks have disjoint edge sets and meet in at most one vertex, which is then a cut vertex. Every cycle lies in a single block. Every edge belongs to exactly one block.

The following standard product formulas also follow from the multiplicativity of the Tutte polynomial under one-point joins; see Ellis-Monaghan and Merino~\cite[Section~3.1]{EM}.

\begin{lemma}\label{lem:blocks}
Let $G$ be a connected graph with $|V(G)|\geq 2$, and let $G_1,\ldots,G_m$ be its blocks. Then
\[
\tau(G)=\prod_{i=1}^m\tau(G_i),
\qquad
\alpha(G)=\prod_{i=1}^m\alpha(G_i),
\qquad
\alpha^*(G)=\prod_{i=1}^m\alpha^*(G_i).
\]
\end{lemma}

\vskip 0.25cm
\begin{proof}[\textbf{Proof of Theorem~\ref{thm:main}}]
The case $G=K_1$ follows from the stated convention. Otherwise, let $G_1,G_2,\ldots,G_m$ be the blocks of $G$. Since $G$ is connected, has at least two vertices and has no bridges, every $G_i$ is $2$-connected. Each $G_i$ is chordal, being an induced subgraph of $G$. By Theorem~\ref{thm:2connected} and Lemma~\ref{lem:blocks}, we have
\[
\alpha(G)\alpha^*(G)
=\prod_{i=1}^m\alpha(G_i)\alpha^*(G_i)
\geq\prod_{i=1}^m\tau(G_i)^2
=\tau(G)^2.
\]
The additive inequality follows from the arithmetic--geometric mean inequality. The original inequality then follows from
$\max\{\alpha(G),\alpha^*(G)\}\geq(\alpha(G)+\alpha^*(G))/2\geq\tau(G)$.
\end{proof}

\section{Concluding remarks}

The proof uses clique trees to reduce the induction step to complete graph attachments with rooted $2$-tree edge decorations. Comparing orientation and spanning tree counts shows that the attachments arising in this decomposition
preserve the multiplicative Merino--Welsh inequality. Decorated $K_4$ attachments along an edge and along a triangle are treated separately. Induction on the number of maximal cliques of order at least four proves the $2$-connected case, and block decomposition extends the result to connected bridgeless chordal graphs.

When two graphs intersect in a complete subgraph, acyclic orientations agreeing on its edges combine to form an acyclic orientation of their union. For complete graph attachments, the spanning tree comparison gives upper bounds
on the number of spanning trees using explicit counts for unions of two complete graphs. A possible extension is to obtain analogous orientation and spanning tree bounds for unions of graphs whose common vertex set need not be a clique.

\section*{Declarations}
\noindent\textbf{Generative AI and AI-assisted technologies.} During the preparation of this work, the authors used ChatGPT (OpenAI) to assist in
exploring potential approaches to the proof of Theorem~\ref{thm:main}. All AI-generated suggestions were verified and refined by the authors, who take
full responsibility for the correctness and originality of the paper.
\smallskip

\noindent\textbf{Funding.} The work was supported by the National Natural Science Foundation of China (Grant No.\ 12271251) and the China Scholarship
Council (CSC No.\ 202606830081).
\smallskip

\noindent\textbf{Conflict of interest.} The authors declare that they have no conflict of interest.
\smallskip

\noindent\textbf{Data and code availability.} No datasets were generated or analyzed during this study.

\end{document}